\documentclass[12pt]{article}

\usepackage[all]{xy}
\usepackage{xypic,amsmath,amsfonts,amsthm,amssymb,amstext,amsbsy,graphicx,comment,epsfig,comment}
\usepackage{mathtools}
\usepackage{mine}
\begin{document}
\define\cx{\langle x \rangle}
\define\lx{\lambda(x)}
\define\lxi{\lambda(x)^{-1}}
\define\lt{\lambda(t)}
\define\lti{\lambda(t)^{-1}}
\define\clx{\langle\lambda(x)\rangle}
\define\lG{\lambda(G)}
\define\rG{\rho(G)}
\define\lN{\lambda(N)}
\define\X[#1]{\{1,x,\dots,x^{#1-1}\}}
\define\Y[#1]{\{t,tx,\dots,tx^{#1-1}\}}
\define\lg{\lambda(g)}
\define\lgi{\lambda(g)^{-1}}
\parskip=0.125in
\title{Enumerating Dihedral Hopf-Galois Structures Acting on Dihedral Extensions}
\date{\today}
\author{Timothy Kohl\\
Department of Mathematics and Statistics\\
Boston University\\
Boston, MA 02215\\
tkohl@math.bu.edu}
\maketitle
\begin{abstract}
The work of Greither and Pareigis details the enumeration of the Hopf-Galois structures (if any) on a given separable field extension. For an extension $L/K$ which is classically Galois with $G=Gal(L/K)$ the Hopf algebras in question are of the form $(L[N])^{G}$ where $N\leq B=Perm(G)$ is a regular subgroup that is normalized by the left regular representation $\lG\leq B$. We consider the case where both $G$ and $N$ are isomorphic to a dihedral group $D_n$ for any $n\geq 3$. Using the normal block systems inherent to the left regular representation of each $D_n$, (and every other regular permutation group isomorphic to $D_n$) we explicitly enumerate all possible such $N$ which arise.
\end{abstract}
\section{Preliminaries}
The general theory of Hopf-Galois extensions can be found in references such as \cite{CS69} and \cite{Sweedler1968} which applies for general extensions of commutative rings. Our focus will be on the case of separable extensions of fields as elucidated in \cite{GreitherPareigis1987}. We give some of the background below.\par
\noindent Given a separable extension of fields $L/K$, and a $K$-Hopf algebra $H$, we say that $L/K$ is $H$-Galois if there exists a $K$-algebra homomorphism $$\mu:H\rightarrow End_{K}(L)$$
such that 
$$\mu(h)(ab)=\sum_{(h)}\mu(h_{(1)})(a)\mu(h_{(2)})(b)$$
where $\Delta(h)=\sum_{(h)}h_{(1)}\otimes h_{(2)}$ and where the fixed ring 
$$L^H=\{x\in L| \mu(h)(x)=\epsilon(h)x\ \forall h\in H\}$$
is precisely $K$ and the induced map $1\otimes\mu:L\otimes H\rightarrow End_K(L)$ is an isomorphism. It is easy to check, for example, that a 'classical' Galois extension $L/K$ with $G=Gal(L/K)$ is Hopf-Galois for the $K$-Hopf algebra $K[G]$. For a general separable extension $L/K$, the work of Greither and Pareigis \cite{GreitherPareigis1987} demonstrates the method one uses to enumerate the Hopf-Galois structures (if any) that may arise. Let $\tilde L$ be the Galois closure of $L/K$ and consider \par
$$
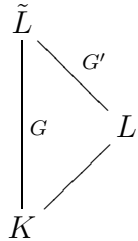
\diagram
        & \tilde L \ddline^{G} \drline^{G'} \\
        & & L \dlline \\
        & K
\enddiagram$$
where $G$ and $G'$ are the relevant Galois groups. If we let $S=G/G'$ then $S$ has a natural $G$ action, by left translation on cosets, which yields an embedding $\lambda:G\rightarrow B=Perm(S)$. A {\it regular} subgroup $N$ of $B$ is one that acts transitively and fixed-point freely. The following shows to enumerate and describe what Hopf algebras (if any) act on $L/K$ to make it Hopf-Galois. 
\begin{theorem}\cite[2.1]{GreitherPareigis1987}
Let $L/K$ be a separable field extension, $S$ and $B$ as above, then the following are
 equivalent:\par
(a) There is a $K$-Hopf algebra $H$ such that $L/K$ is $H$-Galois\par
(b) There is a regular subgroup $N\leq B$ that is normalized by $\lambda(G)\leq B$.\par
\noindent The Hopf algebra in (a) is a $\tilde L$-form of $K[N]$ and can be computed by means of Galois descent.
\end{theorem}
The Hopf algebras themselves are $(\tilde L[N])^{G}$ the fixed ring of $\tilde L[N]$ under the diagonal action of $G$ on $N$ and $\tilde L$ simultaneously. Although structural questions are interesting in their own right, our focus in this discussion is on determining the number and type of those $N$ that arise.\par
\noindent The enumerative side involves determining if there exist regular subgroups $N$ normalized by the image of $G$ under the induced map mentioned above, and we will be considering the case where $L/K$ is already a Galois extension.\par 
If $L/K$ is Galois with group $G$ then $S=G$ and the embedding of $G$ in $B$ is the left regular representation $\lambda:G\rightarrow Perm(G)$. As such, $N=\rho(G)$ where $\rho$ is the {\it right} regular representation itself gives a Hopf Galois structure, specifically $H_{N}=K[N]\cong K[G]$. The reason for this is that $\lambda(G)$ centralizes $\rho(G)$ so the diagonal action reduces to $G$ acting on $L$. The regularity of $\rho(G)$ is clear. We note also that if $G$ is non-abelian then $\lambda(G)$ normalizes itself and so $N=\lambda(G)$ will yield a {\it different} Hopf Galois structure on $L/K$ since $\lambda(G)\neq\rho(G)$ and $H_{\lambda(G)}=L[\lambda(G)]^G\not\cong K[G]\cong H_{\rho(G)}$. In general the group $N$ must have the same cardinality as $G$ by regularity, but need not be isomorphic to $G$ itself. We have the following definitions/conventions for the enumeration of the $N$ which give rise to Hopf-Galois structures.\par
\noindent Definitions:
\begin{itemize}
\item $B=Perm(G)\cong S_{_{|G|}}$
\item $R(G)=\{N \leq B | N \text{ regular},\ \lambda(G) \leq Norm_B(N)\}$
\item $R(G,[M])=\{N\in R(G) |\ N \cong M\}$ 
\item $Hol(N)=Norm_B(N)\cong N\rtimes Aut(N)$ for any $N\in R(G)$
\end{itemize}
where $[M]$ represents the different isomorphism classes of groups of order $|G|$.\par
The determination of $R(G,[M])$ for different pairings $(G,[M])$ has been studied by a number of authors, for example $R(G,[M])$ for $|G|=|M|=pq$ for $p,q$ prime in \cite{Byott2004pq}, $R(S_n,[S_n])$, $R(A_4,[A_4])$ and $R(G,[G])$ for $G$ simple in \cite{CarnahanChilds1999}, $R(C_{p^n},[M])$ in \cite{Kohl1998},  and others. In this paper, we shall restrict our attention to those structures arising due to $N\in R(G,[G])$, for $G=D_n$. We shall show the following.\par
{\bf Theorem:} 
$$
|R(D_n,[D_n])| = \begin{cases}
                       (\frac{n}{2}+2)|\Upsilon_n| \ \ \text{ if }8|n\\
                       (\frac{n}{2}+1)|\Upsilon_n| \ \ \text{ if }4|n \text{ but }8\nmid n\\
                       (n+1)|\Upsilon_n| \ \ \text{ if }2|n\text{ but }4\nmid n\\
                       |\Upsilon_n| \ \ \ \ \ \ \ \ \ \ \ \text{ if }n\text{ odd }\\
         \end{cases}
$$
where $\Upsilon_n=\{u\in U_n\ |\ u^2=1\}$ with $U_n$ the group of units mod $n$.\par
\section{The Left Regular Representations of $D_n$}
We assume that $L/K$ is Galois with group $G=D_n$ and so $B=Perm(D_n)$ which may be presented as
\begin{align*}
D_n &= \{x,t |\ x^n=1, t^2=1, xt=tx^{-1}\} \\
    &= \{1,x,x^2,\dots,x^{n-1},t,tx,tx^2,\dots,tx^{n-1}\}\\
\end{align*}
where $|D_n|=2n$, for $n\geq 3$. \par
With the presentation above, we begin with a number of observation about $D_n$ and its automorphism group. 
\begin{proposition}
\label{AutGamma}
For $n\geq 3$ with $D_n=\{t^ax^b|a\in\mathbb{Z}_{2};b\in\mathbb{Z}_{n} \}$ and letting $U_n=\mathbb{Z}_{n}^{*}$,\par
(a) $Aut(D_n)=\{\phi_{i,j} | i\in\mathbb{Z}_{n};j\in U_n\}$ where \par
\hskip0.45in $\phi_{i,j}(t^ax^b)=t^ax^{ia+jb}$ and $\phi_{{i_2},{j_2}}\circ\phi_{{i_1},{j_1}}=\phi_{i_2+j_2i_1,j_2j_1}$\par
(b) $C=\langle x\rangle$ is a characteristic subgroup of $D_n$\par
(c) $Aut(D_n)\cong Hol(\mathbb{Z}_n)$\par
\end{proposition}
Now, naively, it would seem that the search for any $N\in R(D_n,[D_n])$ would require one to search within the entire ambient symmetric group $B=Perm(D_n)\cong S_{2n}$. However, we shall use structural information about how $\lambda(D_n)$ (and concordantly any $N\in R(D_n,[D_n])$) acts in order to restrict our search to a relatively smaller group within $B$. The beginning of this centers around the left regular representation itself.
\begin{lemma}
For $D_n$ as presented above, the cycle structure of $\lx$ and $\lt$ is given as follows:
\begin{align*}
\lx&=(1,x,x^2,\dots,x^{n-1})(t\ tx^{n-1}\ \dots tx)\\
\lt&=(1\ t)(x\ tx)\cdots (x^{n-1}\ tx^{n-1})\\
\end{align*}
for any $n\geq 3$.
\end{lemma}
The observation that $C=\langle x\rangle$ is characteristic in $D_n$ carries over to the support of the cycles that make up $\lx\in \lambda(C)$ since $\phi_{i,j}(x^b)=x^{jb}$ and $\phi_{i,j}(tx^b)=tx^{i+jb}$ which will be explored fully in the next section.
\section{Wreath Products and Blocks}
\begin{definition}
If $G$ is a permutation group acting on a set $Z$ then a {\it block} for $G$ is a subset $X\subseteq Z$ such that for $g\in G$, $X^{g} = X$ or $X^{g}\cap X=\emptyset$.
\end{definition}
In our example, we shall consider $Z$ as the underlying set of $D_n$ and look at blocks arising from subgroups of $B=Perm(D_n)$ and in particular how regularity ties in with these block structures. These results are obviously standard in the theory of block systems and wreath products, but we present them here for convenience and for the particular application to blocks coming from these regular subgroups isomorphic to $D_n$. \par
Recalling our presentation of $D_n$ define :
\begin{align*}
 X&=\{1,x,x^2,\dots,x^{n-1}\} \\
 Y&=\{t,tx,tx^2,\dots,tx^{n-1}\} \\
 Z&=X\cup Y \\
\end{align*}
where $Z=G$ (as sets) and $B\cong Perm(Z)$. Next, define $\tau_{*}:X\rightarrow Y$ by $\tau_{*}(x^j)=tx^j$ which induces an isomorphism $Perm(X)\rightarrow Perm(Y)$. From here on, we set $B_X=Perm(X)$ and $B_Y=Perm(Y)$ and consider the semi-direct product
$$W(X,Y)=(B_X\times B_Y) \rtimes \langle \tau \rangle$$
where $\tau$  has order 2 and is defined as follows:
\begin{align*}
\tau(\beta)(x)&=\tau_{*}^{-1}(\beta(\tau_{*}(x))) \text{ for }  x\in X \text{ and } \beta\in B_Y \\
\tau(\alpha)(y)&=\tau_{*}(\alpha(\tau_{*}^{-1}(y))) \text{ for } y\in Y \text{ and } \alpha\in B_X \\
\end{align*}
If, for $\alpha\in B_X$ and $\beta\in B_Y$ we denote $\tau(\alpha)=\widehat{\alpha}$ and $\tau(\beta)=\widehat{\beta}$ then $\tau(\alpha,\beta)=(\widehat{\beta},\widehat{\alpha})$. As $B_X\cong B_Y\cong S_{n}$ and $\langle\tau\rangle\cong S_2$ we find that 
$$W(X,Y)\cong S_{n}\wr S_2$$
the wreath product of $S_{n}$ and $S_2$.\par
\noindent Note: As an element of $B$,
\begin{align*}
\tau&=(1,t)(x,tx)\dots(x^{n-1},tx^{n-1})\\
    &=\lambda(t)\\
\end{align*}
Define a map $\delta:W(X,Y)\rightarrow B=Perm(Z)$ by 
$$\delta(\alpha,\beta,\tau^k)(z)=
\begin{cases}
\beta(\tau_{*}(z)),\ k=1,\ z\in X \\ \alpha(z),\ k=0,\ z\in X \\
                              \alpha(\tau_{*}^{-1}(z)),\ k=1,\ z\in Y \\ \beta(z),\ k=0,\ z\in Y \\
\end{cases}
$$
\begin{proposition}
$\delta$ as defined above is an embedding of $W(X,Y)$ as a subgroup of $B$.
\end{proposition}
We need to make a number of other observations about wreath products such as $W(X,Y)$, which we identify with $\delta(W(X,Y))\leq B$. 
\begin{lemma}
If $w\in W(X,Y)$ then either $w(X)=X$ and $w(Y)=Y$ or $w(X)=Y$ and $w(Y)=X$.
\end{lemma}
\begin{proof}
For $k=0$ we have exactly $w(X)=X$ and $w(Y)=Y$ by the definition of $\delta$, the reverse set mappings occur for $k=1$.
\end{proof}
As such, we may regard $W(X,Y)$ as the maximal subgroup of $B$ for which $X$ is a block. 
Although we shall use the above indicated choice of $\tau_{*}$, it is useful to observe the following.
\begin{proposition}
For any two bijections $\tau_{*}$ and $\tau_{*}'$ of $X$ to $Y$, the induced wreath products $W$ and $W'$ are equal as subgroups of $B$.
\end{proposition}
\begin{proof}
Basically we want to, given $(\alpha,\beta,\tau^k)$ in $W$ induced by $\tau_{*}$, find a corresponding $(\alpha',\beta',{\tau'}^{k'})$ which yields the same permutation of $Z$. For $k=0$ we can choose $k'=0$ and simply let $\alpha'=\alpha$ and $\beta'=\beta$. For $k=1$ we have $k'=1$ and we define 
\begin{align*}
\beta'(y)&=\beta(\tau_{*}({\tau_{*}'}^{-1} (y)))\text{ for y}\in Y \\
\alpha'(x)&=\alpha(\tau_{*}^{-1}({\tau_{*}'}(x)))\text{ for x}\in X \\
\end{align*}
One can easily check that $\beta(\tau_{*}(x))=\beta'({\tau_{*}'}(x))$ for $x\in X$ and $\alpha(\tau_{*}^{-1}(y))=\alpha'({\tau'}^{-1}(y))$ for $y\in Y$.
\end{proof}
\begin{definition} For $Z$ such that $|Z|=2n$, a {\it splitting} $\{X,Y\}$ of $Z$ is a partition of $Z$ into two equal size subsets.\end{definition}
\noindent As we can see, for a given splitting $\{X,Y\}$ of $Z$ every bijection $\tau_{*}:X\rightarrow Y$ yields the same subgroup of $B$ which we may denote $W(X,Y;\tau_{*})$ or simply $W(X,Y)$. Note now that for an arbitrary splitting $\{X,Y\}$ that $w\in W(X,Y)$ if and only if $w(X)=X$ and $w(Y)=Y$ or $w(X)=Y$ and $w(Y)=X$.\par
\begin{proposition}
For a given splitting $\{X,Y\}$ and $\sigma\in B$, we have 
$$\sigma W(X,Y)\sigma^{-1}=W(X^{\sigma},Y^{\sigma})$$
where $X^{\sigma}=\sigma(X)$ and $Y^{\sigma}=\sigma(Y)$.
\end{proposition}
\begin{proof}
First, observe that $\{X,Y\}$ a splitting implies that $\{X^{\sigma},Y^{\sigma}\}$ is too. If $w\in W(X,Y)$ with $w(X)=X$ and $w(Y)=Y$ then $\sigma w \sigma^{-1}(X^{\sigma})=\sigma w\sigma^{-1}(\sigma(X))=\sigma w(X)=\sigma(X)=X^{\sigma}$, similarly $\sigma w \sigma^{-1}(Y^{\sigma})=Y^{\sigma}$. Hence $\sigma w \sigma^{-1}\in W(X^{\sigma},Y^{\sigma})$. If $w(X)=Y$ and $w(Y)=X$ then $\sigma w \sigma^{-1}(X^{\sigma})=Y^{\sigma}$ and $\sigma w \sigma^{-1}(Y^{\sigma})=X^{\sigma}$ and again $\sigma w \sigma^{-1}\in W(X^{\sigma},Y^{\sigma})$.
\end{proof}
\begin{corollary}
\label{NW=W}
$Norm_B(W(X,Y))=W(X,Y)$
\end{corollary}
\noindent As a small aside, we can consider, for a given $\{X,Y\}$ the subgroup $S(X,Y)=B_X\times B_Y$ of $B$.
\begin{proposition} $S(X,Y)\triangleleft W(X,Y)$ and, in fact, $Norm_B(S(X,Y))=W(X,Y)$.
\end{proposition}
\begin{proof}
Observe that $\sigma\in B$ is an element of $S(X,Y)$ if and only if $\sigma(X)=X$ and $\sigma(Y)=Y$. If $w\in W(X,Y)$ and $w(X)=X$ and $w(Y)=Y$ then clearly $w\sigma w^{-1}\in S(X,Y)$. Similarly, if $w(X)=Y$ and $w(Y)=X$ then $w\sigma w^{-1}(X)=w\sigma(Y)=X$ and likewise $w\sigma w^{-1}(Y)=Y$.  The proof of the second assertion is identical to that of \ref{NW=W}.
\end{proof}
\noindent Note, $W(X,Y)=S(X,Y)\cup S(X,Y)\tau$ for any $\tau$ induced by $\tau_{*}:X\rightarrow Y$. \par
Before considering the enumeration of $R(D_n,[D_n])$ we shall first consider how regularity and block structure are connected. The following is basically \cite[Theorem 1.6A (i)]{DixonMortimer1996}, the point being that $K$ giving rise to $\{X,Y\}$ is an example of a normal block system. \par
\begin{proposition} If $N\leq B$ is regular then $N\leq W(X,Y)$ if and only if $N$ contains an index 2 subgroup $K$ with $X=K\cdot 1$.
\label{I2K}
\end{proposition}
\begin{proof}
Assume that $1\in X$, and let $K=N\cap S(X,Y)=N_{X}$ the isotropy subgroup of $X$ with respect to the induced action of $N$ on $\{X,Y\}$. As such $[N:K]$=1 or 2, but $[N:K]=1$ is impossible as it would violate the transitivity (regularity) of $N$. That $X=K\cdot 1$ is obvious. The converse is similar, starting with the observation that for a given $K\leq N$ with $N$ regular and $[N:K]$=2, that $X=K\cdot 1$ and $Y=Z-X$ yield a splitting $\{X,Y\}$ with $N\leq W(X,Y)$. 
\end{proof}
The $K$'s which arise are of course normal, but we need the following fact about the normalizers of regular subgroups $N$.
\begin{proposition}
If $N\leq B$ is regular and $N\leq W(X,Y)$ corresponding to $K\leq N$ as above, then $Norm_{B}(N)\leq W(X,Y)$ if and only if $K$ is a characteristic subgroup of $N$.
\label{charK}
\end{proposition}
\begin{proof} The mapping $b: N\rightarrow G$ (given by $b(n)=n\cdot 1$) is a bijection inducing an isomorphism $\phi:Perm(G)\rightarrow Perm(N)$ and if $X=K\cdot 1$, then we may define $b(X)=Ke_{N}=K=\tilde X$, and similarly $\tilde Z=N$ and $\tilde Y=\tilde Z-\tilde X$. In $Perm(N)$ we have $\phi(K)=\lambda(K)\leq \lambda(N)=\phi(N)$ corresponding to $\tilde X=K$. Now, $Hol(N)=Norm_{Perm(N)}(N)=\rho(N)Aut(N)$ and so for $\eta\in Hol(N)$ we have $\eta=\rho(m)\alpha$ for $m\in N$ and $\alpha\in Aut(N)$ and so if $K$ is characteristic then
\begin{align*}
\eta(\tilde X)   &= \rho(m)\alpha(K) \\
	  &=\rho(m)\alpha(K) \\
	  &=\rho(m)(K)\\
	  &=Km^{-1}\\
          &=K\ or\ N-K\ (i.e.\ \tilde X\ or\ \tilde Y)\\
\end{align*}
For the converse observe that $N=K\cup nK$ (for some $n\not\in K$) and so for $\alpha\in Aut(N)\leq Norm_B(N)$ we have $\alpha(K)=K$ or $nK$ which, of course, means $\alpha(K)=K$, so, in fact, $Norm_B(N)\leq W(X,Y)$ implies $Aut(N)\leq S(X,Y)$. 
\end{proof}
Of course, $K$ being characteristic in $N$ corresponds to $K\triangleleft Norm_B(N)$ and since $Norm_B(N)$ acts transitively on $X\cup Y$ then this is yet another instance of a normal block system, where the blocks are the same as those arising from $K\triangleleft N$. And for $\lambda(D_n)$ the block structure is as follows.
\begin{proposition}
\label{GammaWreath}
Given $G=D_n$ as presented above, then:\par
(a) If $n$ is odd, then $\lambda(G)\leq W(X_0,Y_0)$ for exactly one $\{X_0,Y_0\}$.\par
(b) If $n$ is even then $\lambda(G)\leq W(X_i,Y_i)$ for exactly three $\{X_i,Y_i\}$\par
\end{proposition}
\begin{proof}
The underlying set is $\{1,x,\dots,x^{n-1},t,tx,\dots,tx^{n-1}\}$ and 
$$\lx = (1\ x\ \cdots x^{n-1})(t\ tx^{n-1}\ \dots tx)$$
and
$$\lt=(1\ t)(x\ tx)\cdots (x^{n-1}\ tx^{n-1})$$
\noindent For $n$ odd, the claim is that there is exactly one block of size $n$ (equivalently only one splitting yielding a wreath product containing $G$), namely
$$
X_0=\{1,x,\dots,x^{n-1}\}\ and\ Y_0=\{t,tx,\dots,tx^{n-1}\}
$$
That $\lambda(G)\leq W(X_0,Y_0)$ is clear. 
That this is the only possibility follows from the observation that $K=\langle\lx\rangle$ is the only index 2 subgroup of $D_n$.\par
\noindent For $n$ even, we have the following two additional splittings:
\begin{align*}
X_1&=\{1,x^2,\dots,x^{n-2},t,tx^2,\dots,tx^{n-2}\}\\
Y_1&=\{x,x^3,\dots,x^{n-1},tx,tx^3,\dots,tx^{n-1}\}
\end{align*}
and
\begin{align*}
X_2&=\{1,x^2,\dots,x^{n-2},tx,tx^3,\dots,tx^{n-1}\}\\
Y_2&=\{x,x^3,\dots,x^{n-1},t,tx^2,\dots,tx^{n-2}\}
\end{align*}
and that $G$ is contained in the associated wreath products can be verified by considering the actions of $\lx$ and $\lt$. If now $\lambda(G)\leq W(X,Y)$ for some {\it other} splitting $\{X,Y\}$ then assume $1\in X$. If $\lx(X)=X$ then we have $X=X_0$. If $\lx(X)=Y$ then $\lx(Y)=X$ and so $\{1,x^2,\dots,x^{n-2}\}\subseteq X$ and $\{x,x^3,\dots,x^{n-1}\}\subseteq Y$. If $\lt(X)=X$ then $X$ contains $\{t,tx^2,\dots,tx^{n-2}\}$ and so $X=X_1$. If $\lt(X)=Y$ then $Y$ contains $\{tx,tx^3,\dots,tx^{n-1}\}$ and therefore $Y=Y_2$.\par
\end{proof}
Of course, the splittings $\{X_i,Y_i\}$ arise from \ref{I2K} by virtue of the following.
\begin{corollary}
For $n$ odd, $D_n$ has exactly one subgroup of index 2 (hence characteristic), and for $n$ even, 3 subgroups of index 2.
\end{corollary}
\begin{proof}
We observe for $n$ odd that $K_0=\langle x\rangle$ is an index two subgroup of $D_n$ which by \ref{I2K} and \ref{GammaWreath} is unique and therefore characteristic. For $n$ even, we know by \ref{I2K} and \ref{GammaWreath} that there will be exactly three index 2 subgroups, and we can exhibit them easily, namely $K_0=\langle x\rangle$, $K_1=\langle x^2,t\rangle$ and $K_2=\langle x^2,tx^{n-1}\rangle$.
\end{proof}
\begin{corollary}
\label{uniqueHolinW}
For all $n$, $Hol(D_n)$ is a subgroup of $W(X_0,Y_0)$ for a unique $\{X_0,Y_0\}$.
\end{corollary}
\begin{proof}
For $n$ odd, the statement is true by the above corollary together with \ref{charK}. Given our knowledge of $Aut(D_n)$ we can show for $n$ even, that, of the index two subgroups $K_0$, $K_1$ and $K_2$, only $K_0=\langle x\rangle$ is characteristic.
\end{proof}
As to the block structures for any $N\in R(D_n,[D_n])$ we utilize the fact that $\lambda(D_n)$ normalizes any such $N$.
\begin{theorem}
\label{Nwreath}
Let $N\in R(D_n,[D_n])$ with $K$ the characteristic index 2 subgroup of $N$ and $X=K\cdot 1$ (with $Y=Z-X$).\par\vskip0.125in
(a) If $n$ is odd then $X=X_0$.\par\vskip0.125in
(b) If $n$ is even then $X=X_i$ for either $i=0$, $1$, or $2$.
\end{theorem}\par\vskip0.125in
\begin{proof}
Part (a) is a consequence of the fact that $\lambda(G)\leq W(X_0,Y_0)$ uniquely so that $\lambda(G)\leq Norm_B(N)\leq W(X,Y)$ implies $X=X_0$.\par\vskip0.125in
Part (b) is a consequence of the fact that $Norm_B(N)\leq W(X,Y)$ and $\lambda(G)\leq W(X_i,Y_i)$ for $i=0,1,2$ so that $X$ must be $X_i$ for {\it exactly one} such $i$.
\end{proof}
\par\vskip0.125in
\noindent The splitting corresponding to the index 2 characteristic subgroup $K$ of any $N\in R(D_n,[D_n])$ is sufficient to actually determine $N$ itself. To see this, we start by considering the subgroup $K_0=\langle\lambda(x)\rangle\leq\lambda(D_n)$
\begin{proposition}\cite[2.6]{Kohl2015}
\label{strongnorm}
Given $G=D_n$ as presented above, with $C=\lambda(\langle x\rangle)$, (i.e. $K_0$ above) one has $Norm_B(C)=Norm_B(\lambda(G))=Hol(G)$.
\end{proposition}
\noindent Now that we have determined those wreath products that contain $\lambda(D_n)$, we wish to now consider those wreath products that contain a given regular subgroup $N$, and then determine when $\lG\leq W(X,Y)$ and $\lG\leq Norm_B(N)\leq W(X',Y')$ implies $\{X,Y\}=\{X',Y'\}$.  Moreover, we need not worry about the order $2$ generator of $N$.
\begin{proposition}
If $k_Xk_Y$ is product of two disjoint $n$-cycles, then $K=\langle k_Xk_Y\rangle$ is the index 2 characteristic subgroup of exactly one regular subgroup $N\leq B$ where $N\cong D_n$.
\end{proposition}
\begin{proof}
Since $k_Xk_Y$ is a product of two disjoint $n$-cycles where $X=Supp(k_X)$ and $Y=Supp(k_Y)$ then if $N=K\langle\tau\rangle$ we claim that $\tau(X)=Y$ and $\tau(Y)=X$. Since $\tau$ has order 2, it must be a product of $n$ disjoint transpositions by regularity. If $n$ is odd then $\tau(X)=X$ and $\tau(Y)=Y$ is clearly impossible since one of the transpositions would have to contain an element of $X$ and one from $Y$ which would contradict $\tau(X)=X$. If $n$ is even then one {\it could} have $n/2$ transpositions with elements from $X$ and $n/2$ transpositions with elements from $Y$, but what would happen is that the resulting group $\langle k_xk_y,\tau\rangle$ would have fixed points and would therefore not be regular. As such $\tau$ is a product of disjoint transpositions where each transposition contains one element from $X$ and one from $Y$, and where $\tau k_x\tau^{-1}=k_Y^{-1}$ and $\tau k_Y \tau^{-1}=k_x^{-1}$. Specifically if $k_X=(z_1,z_2,\dots,z_n)$ and $k_Y=(z_1',z_2',\dots z_n')$ (whence $k_Y^{-1}=(z_n',z_{n-1}',\dots,z_2',z_1')$) then the only possibilities for $\tau$ are
\begin{align*}
&(z_1,z_n')(z_2,z_{n-1}')(z_3,z_{n-2}')\cdots (z_n,z_1')\\
&(z_1,z_{n-1}')(z_2,z_{n-2}')(z_3,z_{n-3}')\cdots (z_n,z_{n}')\\
&(z_1,z_{n-2}')(z_2,z_{n-3}')(z_3,z_{n-4}')\cdots (z_n,z_{n-1}')\\
&\vdots\\
&(z_1,z_1')(z_2,z_{n}')(z_3,z_{n-1}')\cdots (z_n,z_2')\\
\end{align*}
where each (together with $k_Xk_y$) generate the same group isomorphic to $D_n$.
\end{proof}

\section{Enumerating $K$}
The enumeration of $N\in R(D_n,[D_n])$ is equivalent to the characterization of $K\leq N$ the (cyclic) characteristic subgroup of index 2.\par\vskip0.125in
We divide the analysis between the case where $n$ is odd, versus when $n$ is even. The biggest difference is that when $n$ is odd, any $N\in R(D_n,[D_n])$ must satisfy $N\leq W(X_0,Y_0)$ whereas if $n$ is even, then one may have $N\leq W(X_i,Y_i)$ for $i=0,1,2$ potentially. As it is integral to the determination of $|R(G,[G])|$ we examine the notion of the multiple holomorph of a group, as formulated in \cite{Kohl2015}\par
\begin{proposition}
The collection
$$
\mathcal{H}(G)=\{\text{ regular }N\leq Hol(G)\ |\ N\cong G\text{ and }Hol(N)=Hol(G)\}
$$
is exactly parameterized by $\tau\in T(G)=Norm_B(Hol(G))/Hol(G)$ the multiple holomorph of $G$, namely $\mathcal{H}(G)=\{\tau\lambda(G)\tau^{-1}\ |\tau\in T(G)\}$.
\end{proposition}
The significance of this to the enumeration of $R(G,[G])$ in general is that $\mathcal{H}(G)\subseteq R(G,[G])$ since $\lG\leq Hol(G)=Hol(N)$ for any $N\in\mathcal{H}(G)$. And for $D_n$ we have the following 
\begin{theorem}\cite[Thm. 2.11]{Kohl2015}
For $G=D_n$ we have:
$$|\mathcal{H}(D_n)|=|T(D_n)|=|\Upsilon_n|$$
where $\Upsilon_n=\{u\in U_n\ |\ u^2=1\}$ the units of exponent 2 mod n.
\end{theorem}
The cardinality of $\Upsilon_n$ will play a central role in the enumeration of $R(D_n,[D_n])$ for each $n$. However the arguments differ a bit for the cases where $n$ is odd versus when $n$ is even, and for those $N$ such that $Norm_B(N)\leq W(X_0,Y_0)$ versus those where $Norm_B(N)\leq W(X_i,Y_i)$ for $i=1,2$.  We note a technical fact which will be used in the subsequent theorem.
\begin{lemma}
Let $n$ be even and $v\in\Upsilon_n$:\par\vskip0.125in
(a) if $8\nmid n$ then $gcd(v+1,n)=2$ only if $v=1$\par
(b) if $8|n$ then $gcd(v+1,n)=2$ only if $v=1,\frac{n}{2}+1$
\end{lemma}
\begin{proof}
Since $n=2m$ then $v=2l+1$ and since $v\in \Upsilon_n$ we have $2l^2+2l\equiv 0(mod\ m)$ so that $2l(l+1)=qm$ for some $q$. Since $gcd(v+1,n)=gcd(2l+2,2m)=2$ then $gcd(l+1,m)=1$ which means $(l+1)|q$, that is $q=d(l+1)$. Since $v<n$ then $l+1\leq m$ so therefore $2l(l+1)=d(l+1)m$ which implies $2l=dm$. As such either $d=0$ or $d=1$. If $d=0$ then $2l=0$ so that $l=0$ and $v=1$. If $d=1$ then $2l=m$ and if $m=2k+1$ we have a contradiction. If $d=1$ and $m=2k$ then $2l=2k$ so $l=k$ and $v=2k+1=m+1$ and so $v+1=m+2=2k+2$. But now $gcd(v+1,n)=gcd(2k+2,4k)=2$ which means $gcd(k+1,2k)=1$. However, if $8\nmid n$ then $4\nmid m$ and so $2\nmid k$ which means $k+1$ is even, which means $gcd(k+1,2k)$ cannot equal 1. As such, for $v\in U_n$ to be such that $gcd(v+1,n)=2$ we must have $2|k$ and thus $8|n$ and thus either $v=1$ or $v=\frac{n}{2}+1$.
\end{proof}
Our first segment of the enumeration of $R(D_n,[D_n])$ is the following.
\begin{theorem}
For $G=D_n$, if 
$$R(D_n,[D_n];W(X_i,Y_i))=\{N\in R(D_n,[D_n])\ |\ Norm_B(N)\leq W(X_i,Y_i)\}$$
then\par\vskip0.125in
(a) If $n$ is odd, $|R(D_n,[D_n])|=|R(D_n,[D_n];W(X_0,Y_0))|=|\Upsilon_n|$.\par\vskip0.125in
(b) If $n$ is even then $|R(D_n,[D_n];W(X_0,Y_0)|=\mu_n|\Upsilon_n|$  for 
$$\mu_n=|\{v\in\Upsilon_n\ |\ gcd(v+1,n)=2\}|$$
where $\mu_n=2$ if $8|n$, and otherwise $\mu_n=1$.\par\vskip0.125in
\end{theorem}
\begin{proof}
For the splitting
\begin{align*}
X_0&=\{1,x,\dots,x^{n-1}\}\\
Y_0&=\{t,tx,\dots,tx^{n-1}\}
\end{align*}
we recall that $N\in R(D_n,[D_n])$ implies $N\in R(D_n,[D_n];W(X_0,Y_0))$ automatically if $n$ is odd.\par\vskip0.125in
In this case, if $K\leq N$ is the (unique) subgroup of index 2, we have $K=\langle k\rangle=\langle k_Xk_Y\rangle$ where $Supp(k_X)=X_0$ and $Supp(k_Y)=Y_0$. Since $Supp(k_X)=X_0$ and $Supp(k_Y)=Y_0$ then $k_X(x^i)=x^{k_X(i)}$ and $k_Y(tx^j)=tx^{k_Y(j)}$ so we may, for convenience, identify 
\begin{align*}
k_X&=(x^{i_0},x^{i_1},\dots,x^{i_{n-1}})=(i_0,i_1,\dots,i_{n-1})\\
k_Y&=(tx^{j_0},tx^{j_1},\dots,tx^{j_{n-1}})=(j_0,j_1,\dots,j_{n-1})\\
\end{align*}
where $k_X(i_a)=i_{a+1}$ and $k_Y(j_b)=j_{b+1}$. The question is, what are the possibilities for these two $n$-cycles?\par
We begin by using the fact that $N$ (whence $K$) is normalized by $\lambda(D_n)$ so in particular by $\lambda(t)$ and $\lambda(x)$.\par
We have
\begin{align*}
\lx k\lxi(x^i)&=\lx k(x^{i-1})\\
              &=\lx(x^{k_X(i-1)})\\
              &=x^{k_X(i-1)+1}\\
\end{align*}
and
\begin{align*}
\lx k\lxi(tx^j)&=\lx k(x^{j+1})\\
              &=\lx(tx^{k_Y(j+1)})\\
              &=tx^{k_Y(j+1)-1}\\
\end{align*}
where $\lx k\lxi=k^v=k_X^vk_Y^v$ for some $v\in U_n$ where $v^n=1$. We note that since $\lt k_Xk_Y\lti=k_X^uk_Y^u$ then $\lx\lt k_Xk_Y\lti\lxi=k_X^{uv}k_Y^{uv}$ which, since $|\lx\lt|=2$ implies that $(uv)^2=1$ which means that, in fact, $v^2=1$, namely $v\in\Upsilon_n$. Under the identification
\begin{align*}
k_X&=(i_0,i_1,\dots,i_{n-1})\\
k_Y&=(j_0,j_1,\dots,j_{n-1})\\
i_0&=0\ \ j_0=0\\
\end{align*}
we have $k_X^v=(i_0,i_v,\dots,i_{(n-1)v})$ and $k_Y^v=(j_0,j_v,\dots,j_{(n-1)v})$ and therefore:
\begin{align*}
k_X(i_a-1)&=i_{a+v}-1\\
k_Y(j_b+1)&=j_{b+v}+1\\
\end{align*}
If we assume $i_{rv}=1$ for some $r$ then we have 
$$k_X(i_{rv}-1)=k_X(0)=k_X(i_0)=i_{(r+1)v}-1=i_1$$
but then 
\begin{align*}
k_X(i_1)=k_X(i_{(r+1)v}-1)&=i_{(r+2)v}-1=i_2\\
k_X(i_2)=k_X(i_{(r+2)v}-1)&=i_{(r+3)v}-1=i_3\\
&\dots
\end{align*}                                                                    which implies that $i_{(r+e)v}-i_e=1$ for each $e\in\mathbb{Z}_n$. Similarly, for some $s$, we have $j_{sv}+1=j_0=0$ (i.e. $j_{sv}=-1$) and so a similar inductive argument shows that
$$
j_{(s+e)v}-j_e=-1=n-1
$$
for each $e\in\mathbb{Z}_n$. Normalization by $\lt$ yields
\begin{align*}
\lt k\lti(x^i)&=\lt k(tx^i)\\
              &=\lt(tx^{k_Y(i)})\\
              &=x^{k_Y(i)}\\
\end{align*}
and
\begin{align*}
\lt k\lti(tx^j)&=\lt k(x^j)\\
              &=\lt(x^{k_X(j)})\\
              &=tx^{k_X(j)}\\
\end{align*}
where $\lt k\lti=k^u=k_X^uk_Y^u$ for some $u\in U_n$ where $u^2=1$. What this implies is that $x^{i_{e+u}}=k_X^u(x^{i_e})=x^{k_Y(i_e)}$, and if we again focus on the exponents we get
$$i_{e+u}=k_X^u(i_{e})=k_Y(i_e)$$
so we can consider what happens with $e=0,1,\dots$ (recalling that $i_0=j_0=0$ and that $k_Y(j_f)=j_{f+1}$) and we get
\begin{align*}
i_{u}&=k_Y(i_0)=k_Y(j_0)=j_1\\
i_{2u}&=k_Y(i_u)=k_Y(j_1)=j_2\\
i_{3u}&=k_Y(i_{2u})=k_Y(j_2)=j_3\\
&\vdots
\end{align*}
namely $j_{f}=i_{uf}$ for each $f\in\mathbb{Z}_n$, and since $u^2=1$ we can write this as  $j_{uf}=i_{f}$ too. So to summarize so far, we have $n$-cycles $(i_0,\dots,i_{n-1})$ and $(j_0,\dots,j_{n-1})$ where the $i's$ and $j's$ satisfy the following relations
\begin{align*}
i_{0}&=0\\
j_{0}&=0\\
i_{rv}&=1\text{ for some $r$}\\
j_{sv}&=-1\text{ for some $s$}\\
i_{(r+e)v}-i_e&=1\text{ for each $e\in\mathbb{Z}_n$}\\
j_{(s+e)v}-j_e&=-1\text{ for each $e\in\mathbb{Z}_n$}\\
j_{e}&=i_{ue}\text{ for each $e\in\mathbb{Z}_n$}\\
\end{align*}
where $u^2=1$ and $v^2=1$. So the question is what are the solutions of this system of equations, as these determine the possibilities for $k=k_Xk_Y$.\par
The first simplification we can make is that the relation $i_{ue}=j_e$ implies that the values of $j_g$ are completely determined by $i_e$ for $e\in\mathbb{Z}_n$ since $u$ is a unit. As such, we only need to deal with the solution(s) of the equations involving the $i_e$.\par 
The second simplification is to show that, in fact, $r,s\in U_n$. If $r\not\in U_n$ then for some $m<n$ we have $mr\equiv\ 0\ (mod\ n)$. And from the relation $i_{(r+e)v}-i_e=1$ we have
\begin{align*}
i_{rv}-i_0&=1\text{ [$e=0$]}\\
i_{2rv}-i_{rv}&=1\text{ [$e=r$]}\\
i_{3rv}-i_{2rv}&=1\text{ [$e=2r$]}\\
&\vdots\\
i_{mrv}-i_{(m-1)rv}&=1\text{ [$e=(m-1)r$]}\\
\end{align*}
Looking at the left and right hand sides, we see that the indices $\{0,rv,\dots,(m-1)rv\}$ and $\{0,r,\dots,(m-1)r\}$ are equal since $v\in U_n$. As such, if we add these $m$ equations we get
$$0=(\sum_{e=0}^{m-1} i_{erv})-(\sum_{f=0}^{m-1} i_{fr})=m$$
in $\mathbb{Z}_n$ which is impossible since $m<n$. So we conclude that in fact $r\in U_n$ and similarly $s\in U_n$ as well. The next task is to narrow down the possibilities for $v\in \Upsilon_n$\par\vskip0.125in
If $n$ is odd then $v^n=1$ together with $v^2=1$ immediately implies that $v=1$. If $n$ is even then we can use the $v^2=1$ relation as follows. From $i_{(r+e)v}-i_e=1$, $i_0=0$, $i_{rv}=1$ we obtain 
\begin{align*}
i_{rv+rv^2}-i_{rv}&=1\text{ [i.e. $i_{rv+rv^2}=2$]}\\
i_{rv+rv^2+rv^3}-i_{rv+rv^2}&=1\text{ [i.e. $i_{rv+rv^2+rv^3}=3$]}\\
&\vdots\\
i_{rv+rv^2+\dots+rv^e}&=e
\end{align*}
and similarly we have $j_{sv+sv^2+\dots+sv^f}=n-f$. In $i_{rv+rv^2+\dots+rv^e}=e$ we look at the index $rv+rv^2+\dots+rv^e$ and realize that
$$
rv+rv^2+\dots+rv^e=\begin{cases} fr(v+1)\ \ \ \ \ \ \ \text{ if }e=2f \\ fr(v+1)+rv\ \text{ if }e=2f+1\end{cases}
$$
and for the system
\begin{align*}
i_{0}&=0\\
i_{rv}&=1\text{ for some $r$}\\
i_{(r+e)v}-i_e&=1\text{ for each $e\in\mathbb{Z}_n$}\\
\end{align*}
the solutions we seek are those for which all $i_g$ are distinct. As we just saw $i_{rv+rv^2+\dots+rv^e}=e$ for each $e\in\mathbb{Z}_n$ which can be simplified to
\begin{align*}
i_{fr(v+1)}&=2f\text{\ \ \ \ \ \ if }e=2f\\
i_{fr(v+1)+rv}&=2f+1\text{ if }e=2f+1\\
\end{align*}
for $f\in\{0,\dots,\frac{n}{2}-1\}$. So in order that each $i_g$ is distinct we consider whether 
\begin{align*}
f_1r(v+1)&=f_2r(v+1)\\
f_1r(v+1)+rv&=f_2r(v+1)+rv\\
\end{align*}
which is equivalent to $fr(v+1)=0$.\par\vskip0.125in
Since $r$ is a unit then this is equivalent to $f(v+1)=0$ (mod $n$). In $\mathbb{Z}_n$ one has $|v+1|=\frac{n}{gcd(v+1,n)}$ which means $|v+1|=n/2$ if and only if $gcd(v+1,n)=2$ and therefore that $fr(v+1)=0$ only when $f=0$.  (i.e. $f=n/2$) so only when $e=2f=0$ which is consistent with $i_0=0$ and for $e=2f+1=1$ is consistent with $i_{rv}=1$.\par
So for the solutions of
\begin{align*}
i_{0}&=0\\
i_{rv}&=1\text{ for some $r\in U_n$}\\
j_{sv}&=-1\text{ for some $s\in U_n$}\\
i_{(r+e)v}-i_e&=1\text{ for each $e\in\mathbb{Z}_n$}\\
j_{(s+e)v}-j_e&=-1\text{ for each $e\in\mathbb{Z}_n$}\\
j_{e}&=i_{ue}\text{ for each $e\in\mathbb{Z}_n$}\\
\end{align*}
for a given $u\in\Upsilon_n$, and pair $(r,s)\in U_n\times U_n$, we must have $s=-ur$ since $j_{e}=i_{ue}$. If $8\nmid n$ then $v=1$ only, and if $8|n$ $v=1,\frac{n}{2}+1$ and so we have overall
$$
|\Upsilon_n|\cdot \phi(n)\cdot\mu_n
$$
distinct $k_Xk_Y$, which yields $|\Upsilon_n|\cdot\mu_n$ distinct $K=\langle k_Xk_Y\rangle$, and so that many $N\in R(G,[G])$ where $Norm_B(N)\leq W(X_0,Y_0)$.
\end{proof}
This completes the analysis for the case where $Norm_B(N)\leq W(X_0,Y_0)$.\par
We have shown that if $8\nmid n$ and $N$ has block structure $\{X_0,Y_0\}$ then $N\in\mathcal{H}(G)$. Note, this corresponds to $v=1$ only, and for $8|n$ the $v=\frac{n}{2}+1$ possibility yields the other $|\Upsilon_n|$ different $N\in R(G,[G])$ which do {\it not} lie in $\mathcal{H}(G)$.\par
For $n$ even, the situation is a bit more complicated, but can be understood in terms of the other block structures $\{X_1,Y_1\}$ and $\{X_2,Y_2\}$. We will need the following in the subsequent theorem.
\begin{lemma}
\label{autblock}
The automorphism $\phi_{(1,1)}\in Aut(D_n)$, where $\phi(x^b)=x^b$ and $\phi(tx^b)=tx^{b+1}$ has the property that $\phi(X_1)=X_2$, $\phi(X_2)=X_1$ and that $\phi(Y_1)=Y_2$ and $\phi(Y_2)=Y_1$, and also $\phi(X_0)=Y_0$ and $\phi(Y_0)=X_0$.
\end{lemma}
And for the enumeration of those $N$ such that $Norm_B(N)\leq W(X_i,Y_i)$ for $i=1,2$ we also need the following modest, yet important, fact.
\begin{lemma}
\label{DinHolC}
If $n=2m$ is even and $C_n=\langle\sigma\rangle$ is cyclic of order $n$ then $Hol(C_n)$ contains exactly one regular subgroup $D=\langle r,f\ |\ r^{m}=f^2=1\ rf=fr^{-1}\rangle\cong D_{m}$ such that $r\sigma=\sigma r$, and moreover one has that $f\sigma f^{-1}=\sigma^{-1}$.
\end{lemma}
\begin{proof}
The elements of $Hol(C_n)$ are of the form $(\sigma^i,u)$ where $i\in\mathbb{Z}_n$ and $u\in U_n$ where $(\sigma^i,u)(\sigma^j,v)=(\sigma^{i+uj},uv)$, and these act (as permutations) on the elements of $C_n$ as $(\sigma^i,u)(\sigma^k)=\sigma^{i+uk}$. As such, $(\sigma^i,u)$ acts without fixed points provided $i\not\in\langle 1-u\rangle$. Moreover, $(\sigma^i,u)(\sigma^j,v)(\sigma^{i},u)^{-1}=(\sigma^{i+uj-iv},1)$ so any element commuting with $\sigma$ must lie in $\langle(\sigma,1)\rangle$ so we may assume $r=(\sigma^2,1)$. As to the order 2 generator $f=(\sigma^i,u)$, in order that $frf^{-1}=r^{-1}$ we must have that $u=-1$, where $i\not\in \langle 2\rangle$. As such, $\langle r,f\rangle$ is readily seen to be the only regular subgroup of $Hol(C_n)$ isomorphic to $D_m$, and one sees that $f(\sigma,1)f^{-1}=(\sigma^{-1},1)$.
\end{proof}

\begin{theorem}
For $G=D_n$ where $n$ is even then
$$|R(D_n,[D_n];W(X_i,Y_i))|=\dsize\frac{\frac{n}{2}\cdot|\Upsilon_n|\cdot\phi(\frac{n}{2})}{\phi(n)}$$
for $i=1,2$.
\end{theorem}
\begin{proof}
If $n$ is even then a given $N\in R(G,[G])$ is such that $Norm_B(N)\leq W(X_i,Y_i)$ for exactly one $i\in\{0,1,2\}$. The case where $Norm_B(N)\leq W(X_0,Y_0)$ has just been covered. Let's consider $Norm_B(N)\leq W(X_1,Y_1)$ where 
\begin{align*}
X_1&=\{1,x^2,\dots,x^{n-2},t,tx^2,\dots,tx^{n-2}\}\\
Y_1&=\{x,x^3,\dots,x^{n-1},tx,tx^3,\dots,tx^{n-1}\}
\end{align*}
which means $N$'s characteristic two subgroup $K$ is of the form $\langle k_xk_Y\rangle$ where $Supp(k_x)=X_1$ and $Supp(k_Y)=Y_1$. As such we have
\begin{align*}
k_X&=(t^{a_0}x^{b_0},t^{a_1}x^{b_1},\dots,t^{a_{n-1}}x^{b_{n-1}})\\
k_Y&=(t^{c_0}x^{d_0},t^{c_1}x^{d_1},\dots,t^{c_{n-1}}x^{d_{n-1}})\\
\end{align*}
where $a_e,c_e\in\{0,1\}$ and $b_e\in\{0,2,\dots,n-2\}$ and $d_e\in\{1,3,\dots,n-1\}$ where each even number $b_e$ appears twice, and each odd number $d_e$ appears twice, and similarly, half of the $a_e$ are $0$ and half are $1$ and similarly for $c_e$. \par
We can assume that $(a_0,b_0)=(0,0)$ and $(c_0,d_0)=(0,1)$ and that $(a_r,b_r)=(1,0)$  and $(c_s,d_s)=(1,1)$ for some $r,s$ since 
$1,t\in Supp(k_x)$ and $x,tx\in Supp(k_Y)$. The idea then will be to again determine equations amongst the $a_f,b_f,c_f,d_f$ whose solutions govern the potential generators of any such $K\leq N$ characteristic (of index 2) for $N \in R(G,[G])$. We have that $\lambda(x)$ and $\lambda(t)$ must normalize $K$ since $K$ is characteristic in $N$. \par\vskip0.125in
Since 
\begin{align*}
\lambda(t)&=(1,t)(x,tx)\dots(x^{n-1},tx^{n-1})\\
\lambda(x)&=(1,x,\dots,x^{n-1})(t,tx^{n-1},\dots,tx)
\end{align*}
we have that $\lambda(t)(X_1)=X_1$ and $\lambda(t)(Y_1)=Y_1$ while $\lambda(x)(X_1)=Y_1$ and $\lambda(x)(Y_1)=X_1$ and so
\begin{align*}
\lambda(t)k_X\lambda(t)&=k_X^u\\
\lambda(t)k_Y\lambda(t)&=k_Y^u\text{ for some $u\in\Upsilon_n$}\\
&\\
\lambda(x)k_X\lambda(x)^{-1}&=k_Y^v\\
\lambda(x)k_Y\lambda(x)^{-1}&=k_X^v\text{ for some $v\in U_n$ where $v^n=1$}\\
\end{align*}
where, since $|\lt\lx|=2$, we must have that $v^2=1$, i.e. $v\in\Upsilon_n$.\par
\noindent We can also show that $u=-1$. To see this we consider $k_X$ which lies in $B_{X_1}=Perm(X_1)$ and note that 
$$\lt|_{X_1}=(1,t)(x^2,tx^2)\cdots(x^{n-2},tx^{n-2})\in B_{X_1}$$
where now $\lt|_{X_1}$ lies in $Hol(\langle k_X\rangle)=Norm_{B_{X_1}}(\langle k_X\rangle)\cong Hol(C_n)$. Similarly, since $\lx k_X\lx^{-1}=k_Y^v$ and  $\lx k_Y\lx^{-1}=k_X^v$ then $\lx^2 k_X\lx^{-2}=k_X$ and since 
$$\lx^2=(1,x^2,\dots,x^{n-2})(t,tx^{n-2},\dots,tx^{2})(x,x^3,\dots,x^{n-1})(tx^{n-1},tx^{n-3},\dots,tx)$$
then
\begin{align*}
\lx^2|_{X_1}&=(1,x^2,\dots,x^{n-2})(t,tx^{n-2},\dots,tx^{2})\\
\lx^2|_{Y_1}&=(x,x^3,\dots,x^{n-1})(tx^{n-1},tx^{n-3},\dots,tx)\\
\end{align*}
and $\lx^2 k_X\lx^{-2}=k_X$ and therefore $\lx^2|_{X_1}\in Hol(\langle k_X\rangle)$ as well. By \ref{DinHolC} $D=\langle \lx^2|_{X_1},\lt|_{X_1}\rangle$ is the (unique) regular copy of $D_{\frac{n}{2}}$ in $Hol(\langle k_X\rangle)$ whose order $n/2$ generator centralizes $k_X$ which means that $\lt|_{X_1}$ conjugates $k_X$ to $k_X^{-1}$ and similarly $\lx|_{Y_1}$ conjugates $k_Y$ to $k_Y^{-1}$, i.e. $u=-1$.\par
\noindent The other consequence of \ref{DinHolC} is that $\langle\lx^2|_{X_1}\rangle=\langle k_X^2\rangle$ and similarly that $\langle\lx^2|_{Y_1}\rangle=\langle k_Y^2\rangle$ so that for units $w_1,w_2\in U_{\frac{n}{2}}$ we have $k_X=(\lx^2|_{X_1})^{w_1}$ and $k_Y=(\lx^2|_{Y_1})^{w_2}$. Furthermore, since $1\in Supp(k_X)$ and $x\in Supp(k_Y)$ then 
\begin{align*}
(t^{a_0}x^{b_0},t^{a_2}x^{b_2}\dots,t^{a_{n-2}}x^{b_{n-2}})&=(1,x^{2w_1},\dots,x^{(n-2)w_1})\ \  (*)\\
(t^{a_1}x^{b_1},t^{a_3}x^{b_3}\dots,t^{a_{n-1}}x^{b_{n-1}})&=(t,tx^{(n-2)w_1},\dots,tx^{2w_1})\\
(t^{c_0}x^{d_0},t^{c_2}x^{d_2}\dots,t^{c_{n-2}}x^{d_{n-2}})&=(x,x^{2w_2+1},\dots,x^{(n-2)w_2+1})\ \ (*)\\
(t^{c_1}x^{d_1},t^{c_3}x^{d_3}\dots,t^{c_{n-1}}x^{d_{n-1}})&=(tx,tx^{(n-2)w_2+1},\dots,tx^{2w_2+1})\\
\end{align*}
and so for each $e\in\mathbb{Z}_{\frac{n}{2}}$ 
\begin{align*}
a_{2e}&=0 & a_{2e+1}&=1\\
c_{2e}&=0 & c_{2e+1}&=1\\
\end{align*}
and also, since $t^{a_0}x^{b_0}=1$ and $t^{c_0}x^{d_0}=x$, that the cycles labelled (*) are in the same sequence which implies that 
\begin{align*}
b_{2e}&=2ew_1\\
d_{2e}&=2ew_2+1\\
\end{align*}
for each $e\in\mathbb{Z}_{\frac{n}{2}}$. As to the other $n/2$-cycles above, corresponding to $b_{2e+1}$ and $d_{2e+1}$, we defined $r$ and $s$ to be those indices such that $t^{a_r}x^{b_r}=t$ and $t^{c_s}x^{d_s}=tx$ which means that $r,s$ are odd and that
\begin{align*}
(t^{a_r}x^{b_r},t^{a_{r+2}}x^{b_{r+2}}\dots,t^{a_{r+n-2}}x^{b_{r+n-2}})&=(t,tx^{(n-2)w_1},\dots,tx^{2w_1})\\
(t^{c_s}x^{d_s},t^{c_{s+2}}x^{d_{s+2}}\dots,t^{c_{s+n-2}}x^{d_{s+n-2}})&=(tx,tx^{(n-2)w_2+1},\dots,tx^{2w_2+1})\\
\end{align*}
which means
\begin{align*}
b_{r+2e}&=-2ew_1\\
d_{s+2e}&=-2ew_2+1\\
\end{align*}
for $e\in\mathbb{Z}_{\frac{n}{2}}$. So the question is, what is the relationship between $r$ and $s$ and $w_1$ and $w_2$?
Conjugation by $\lambda(x)$ yields
\begin{align*}
(t^{a_0}x^{b_0+(-1)^{a_0}},t^{a_1}x^{b_1+(-1)^{a_1}},\dots,t^{a_{n-1}}x^{b_{n-1}+(-1)^{a_{n-1}}})&=(t^{c_0}x^{d_0},t^{c_v}x^{d_v},\dots,t^{c_{(n-1)v}}x^{d_{(n-1)v}})\\
(t^{c_0}x^{d_0+(-1)^{c_0}},t^{c_1}x^{d_1+(-1)^{c_1}},\dots,t^{c_{n-1}}x^{d_{n-1}+(-1)^{c_{n-1}}})&=(t^{a_0}x^{b_0},t^{a_v}x^{b_v},\dots,t^{a_{(n-1)v}}x^{b_{(n-1)v}})\\
\end{align*}
Here, $(a_0,b_0)=(0,0)$ yields $(a_0,b_0+(-1)^{a_0})=(0,1)=(c_0,d_0)$ so the first equation directly yields that
$$
d_{fv}=b_f+(-1)^{a_f}
$$
for each $f\in\mathbb{Z}_n$. So for $f=2e$ we have $d_{2ev}=2evw_2+1$ and $b_{2e}+(-1)^{a_{2e}}=2ew_1+1$ which means that $2evw_1\equiv 2ew_2(mod\ n)$ for each $e$, so in particular $2vw_2\equiv 2w_1\ (mod\ n)$, which yields 
$$vw_2\equiv w_1\ (mod\ \frac{n}{2})\leftrightarrow w_2\equiv vw_1\ (mod\ \frac{n}{2}) $$
and since $(c_s,d_s)=(1,1)$ then $(c_s,d_s+(-1)^{c_s})=(1,0)=(a_r,b_r)$ which means that
\begin{align*}
&(t^{c_0}x^{d_0+(-1)^{c_0}},t^{c_1}x^{d_1+(-1)^{c_1}},\dots,t^{c_{n-1}}x^{d_{n-1}+(-1)^{c_{n-1}}})=\\
&(t^{a_0}x^{b_0},t^{a_1}x^{b_1},\dots,t^{a_{(n-1)}}x^{b_{(n-1)}})^v=\\
&(t^{a_r}x^{b_r},t^{a_{r+1}}x^{b_{r+1}},\dots,t^{a_{r+(n-1)}}x^{b_{r+(n-1)}})^v\\
&\downarrow\\
&(t^{c_s}x^{d_s+(-1)^{c_s}},t^{c_{s+1}}x^{d_{s+1}+(-1)^{c_{s+1}}},\dots,t^{c_{s+n-1}}x^{d_{s+n-1}+(-1)^{c_{s+n-1}}})=\\
&(t^{a_r}x^{b_r},t^{a_{r+v}}x^{b_{r+v}},\dots,t^{a_{r+(n-1)v}}x^{b_{r+(n-1)v}})\\
\end{align*}
which yields
$$
d_{s+f}+(-1)^{c_{s+f}}=b_{r+fv}
$$
which also implies that $vw_1\equiv w_2\ (mod\ \frac{n}{2})$.\par
So a given
\begin{align*} 
K&=\langle k_Xk_Y \rangle\\
 &=\langle(t^{a_0}x^{b_0},t^{a_1}x^{b_1},\dots,t^{a_{n-1}}x^{b_{n-1}})(t^{c_0}x^{d_0},t^{c_1}x^{d_1},\dots,t^{c_{n-1}}x^{d_{n-1}})\rangle\\
\end{align*}
being normalized by $\lambda(G)$ implies that the following system of equations must be satisfied for each $e\in\mathbb{Z}_n$
\begin{align*}
a_{2e}&=0& a_{2e+1}&=1\\
c_{2e}&=0& c_{2e+1}&=1\\
b_{2e}&=2ew& d_{2e}&=2evw+1\\
b_{r+2e}&=-2ew& d_{s+2e}&=-2evw+1\\
d_{fv}&=b_f+(-1)^{a_f} & b_{r+fv}&=d_{s+f}+(-1)^{c_{s+f}}\\
\end{align*}
for each $e\in\mathbb{Z}_{n/2}$ and $f\in\mathbb{Z}_n$ where $v\in\Upsilon_n$ and $w\in U_{\frac{n}{2}}$. As to the relationship between $r$ and $s$ we observe that $d_{rv}=b_r+(-1)^{a_r}=0-1=-1$ and as, $rv$ is odd, then $s+2e=rv$ for some $e$, which means $2e=rv-s$, which, in turn, implies that $d_{s+2e}=-(rv-s)vw+1$. As such
$$
r-sv=2w^{-1}
$$
which therefore decreases the degrees of freedom, in that $r=sv+2w^{-1}$ where $s\in\mathbb{Z}_n-\langle 2\rangle$.
And so we have $\delta_n$ possible $k_Xk_Y$ where
\begin{align*}
\delta_n&=|\{(s,v,w)\in (\mathbb{Z}_n-\langle 2\rangle)\times \Upsilon_n\times U_{\frac{n}{2}}\}|\\
       &=\frac{n}{2}\cdot|\Upsilon_n|\cdot \phi(\frac{n}{2})\\
       &=\begin{cases} \frac{n}{2}\cdot|\Upsilon_n|\cdot \phi(n) \ \ 4\nmid n\\
                       \frac{n}{2}\cdot|\Upsilon_n|\cdot \frac{\phi(n)}{2} \ \ 4|n
         \end{cases}
\end{align*}
and so $\frac{\delta_n}{\phi(n)}$ possible $K$ which therefore enumerates the $N\in R(D_n,[D_n])$ where $Norm_B(N)\leq W(X_1,Y_1)$.\par
For those $N\in R(D_n,[D_n])$ where $Norm_B(N)\leq W(X_2,Y_2)$ we can utilize \ref{autblock} to establish a bijection
$$
R(D_n,[D_n];W(X_1,Y_1))\overset{\Phi}{\longrightarrow} R(D_n,[D_n];W(X_2,Y_2))$$
as follows. For $\phi_{1,1}$ we have $\phi_{(1,1)}W(X_i,Y_i){\phi_{(1,1)}^{-1}}=W(\phi_{(1,1)}(X_i),\phi_{(1,1)}(Y_i))$ and for a given $N\in R(D_n,[D_n])$ one has that $Norm_B(N)$ is contained in $W(X_i,Y_i)$ for exactly one $\{X_i,Y_i\}$ so we have 
$$|R(D_n,[D_n];\{X_1,Y_1\})|=|R(D_n,[D_n];\{X_2,Y_2\})|$$
which establishes the count in the statement of the theorem.
\end{proof}
\noindent Gathering the results for all $n$, we see that enumeration is dependent on whether $2|n$, $4|n$ or $8|n$ as summarized below:
$$
|R(D_n,[D_n])| = \begin{cases}
                       (\frac{n}{2}+2)|\Upsilon_n| \ \ \text{ if }8|n\\
                       (\frac{n}{2}+1)|\Upsilon_n| \ \ \text{ if }4|n \text{ but }8\nmid n\\
                       (n+1)|\Upsilon_n| \ \ \text{ if }2|n\text{ but }4\nmid n\\
                       |\Upsilon_n| \ \ \ \ \ \ \ \ \ \ \ \text{ if }n\text{ odd }\\
         \end{cases}
$$

The author wishes to thank Griff Elder and the University of Nebraska at Omaha for the hosting of the Hopf Algebras and Galois Module Theory workshop, which has provided an environment which spurred the development of this (and other) works over the last four years.

\bibliography{Dihedral}
\bibliographystyle{plain}
\end{document}